\documentclass[12pt,reqno]{amsart}
\usepackage{amssymb}
\usepackage{bbm}
\usepackage{mathrsfs}
\usepackage[cmtip,arrow]{xy}
\usepackage{pb-diagram,pb-xy}

\newcommand{\zz}{\mathbb{Z}}

\newcommand{\qq}{\mathbb{Q}}

\newcommand{\A}{\mathcal{A}}
\newcommand{\ints}{\mathcal{O}}

\newcommand{\dnd}{\nmid}
\newcommand{\isom}{\cong}

\newcommand{\ord}{\mathrm{ord}}

\newcommand{\Gal}{\mathrm{Gal}}
\newcommand{\N}{\mathrm{N}}
\newcommand{\p}{\mathfrak{p}}

\newtheorem{thm}{Theorem}
\newtheorem{lma}{Lemma}

\theoremstyle{remark}
\newtheorem*{rmk}{Remark}

\usepackage{fullpage}
\begin{document}
\title{A Barban-Davenport-Halberstam asymptotic for number fields\\
(Appeared in \textit{Proceedings of the American Mathematical Society})}

\author{Ethan Smith}
\address{
Department of Mathematical Sciences\\
Michigan Technological University\\
1400 Townsend Drive\\
Houghton, MI 49931-1295
}
\email{ethans@mtu.edu}
\urladdr{www.math.mtu.edu/~ethans}

\begin{abstract}
Let $K$ be a fixed number field, and assume that $K$ is Galois over $\qq$.
Previously, the author showed that when estimating the number of prime ideals 
with norm congruent to $a$ modulo $q$ via the Chebotar\"ev Density Theorem, 
the mean square error in the approximation 
is small when averaging over all $q\le Q$ and all 
appropriate $a$.  In this article, we replace the upper bound by an 
asymptotic formula. 
The result is related to the classical Barban-Davenport-Halberstam Theorem 
in the case $K=\qq$.
\end{abstract}

\maketitle

\section{Introduction}
One of the great results of the 1960s concerning the distribution of primes 
is that ``on average" they are well-distributed in arithmetic 
progressions.  In particular, Barban~\cite{Bar:1964} and, 
independently, Davenport and Halberstam~\cite{DH:1966,DH:1968} showed 
that the square of the error in the Prime Number Theorem for primes in 
arithmetic progressions is small on average.
More precisely, given positive integers $a$ and $q$, 
we define the weighted prime counting function $\theta(x;q,a)$ by 
\begin{equation*}
\theta(x;q,a):=\sum_{\substack{p\le x\\ p\equiv a\pmod q}}\log p.
\end{equation*}
The Prime Number Theorem for 
primes in arithmetic progressions states that if $\gcd(a,q)=1$, then
\begin{equation}
\theta(x;q,a)\sim\frac{x}{\varphi(q)},
\end{equation}
where $\varphi(q):=\#\{1\le a\le q: \gcd(a,q)=1\}$ 
is Euler's $\varphi$-function.
The Barban-Davenport-Halberstam Theorem (see~\cite{Dav:1980})
states that, for any fixed $M>0$,
\begin{equation}\label{original BDH}
\sum_{q\le Q}\sum_{\substack{a=1\\ \gcd(a,q)=1}}^q
\left(\theta(x;q,a)-\frac{x}{\varphi(q)}\right)^2
\ll xQ\log x,
\end{equation}
provided that $x(\log x)^{-M}\le Q\le x$.
Later, Montgomery~\cite{Mon:1970} and Hooley~\cite{Hoo:1975} each gave  
asymptotic formulations of this result valid for various ranges of $Q$.  
Hooley's method starts with the inequality~\eqref{original BDH}, and 
so at least implicitly relies on the large sieve.
Montgomery's method, however, is based on a 
result of Lavrik~\cite{Lav:1960} concerning the distribution of twin primes. 

With applications in mind, there have been several generalizations of this 
result to the integers of a number field.  See~\cite{Hin:1981, Wil:1969} for 
example.  In~\cite{Smi}, the author considered yet another generalization 
of~\eqref{original BDH} concerning the distribution of prime ideals of a 
number field.  See Theorem~\ref{bdh_gen} below.
In the present article, we are concerned with the appropriate asymptotic 
formulation.  See Theorem~\ref{asymp}.

\section{Statement of Main Theorem}

Let $K$ be a fixed number field.  We are concerned with the 
error in estimating sums of the form
\begin{equation*}
\theta_K(x;q,a):=\sum_{\substack{\N\p\le x,\\ \N\p\equiv a\pmod q}}\log\N\p
\end{equation*}
via the Chebotar\"ev Density Theorem.  Here, as usual, $\p$ denotes a 
prime ideal of the ring of integers $\ints_K$, and 
$\N\p:=\#(\ints_K/\p)$ denotes its norm.

Let $\zeta_q$ be a primitive $q$-th root of unity, and let $G_q$ denote 
the image of the natural map
\begin{equation*}
\begin{diagram}
\node{\Gal(K(\zeta_q)/K)}\arrow{e,J}
\node{\Gal(\qq(\zeta_q)/\qq)}\arrow{e,t}{\sim}
\node{(\zz/q\zz)^*.}
\end{diagram}
\end{equation*}
In this case, the Frobenius substitution is determined by the value $\N\p$ 
modulo $q$; and the Chebotar\"ev Density Theorem implies that
if $a\in G_q$, then
\begin{equation}\label{Cheb on theta}
\theta_K(x;q,a)\sim\frac{x}{\varphi_K(q)},
\end{equation}
where we have made the definition $\varphi_K(q):=\#G_q=\#\Gal(K(\zeta_q)/K)$.

If we assume further that $K/\qq$ is a Galois extension, then we have
the following corollary of Goldstein's generalization of the 
Siegel-Walfisz Theorem~\cite{Gol:1970}.  
If $a\in G_q$, then for any fixed $M>0$, 
\begin{equation}\label{gen SW}
\theta_K(x;q,a)
=\frac{x}{\varphi_K(q)}+O\left(\frac{x}{(\log x)^M}\right),
\end{equation}
provided that $q\le(\log x)^{M}$.
The following average error result is the main theorem of~\cite{Smi}, 
where we continue to assume that our number field $K$ is a Galois 
extension of $\qq$.
\begin{thm}\label{bdh_gen}
For a fixed $M>0$,
\begin{equation*}
\sum_{q\le Q}\sum_{a\in G_q}\left(
\theta_K(x;q,a)-\frac{x}{\varphi_K(q)}
\right)^2
\ll xQ\log x
\end{equation*}
if $x(\log x)^{-M}\le Q\le x$.
\end{thm}
\begin{rmk}
To be precise, the main theorem of~\cite{Smi} is stated in terms of 
\begin{equation*}
\psi_K(x;q,a):=\sum_{\substack{\N\p^m\le x,\\ \N\p^m\equiv a\pmod q}}\log\N\p.
\end{equation*}
As usual, the statement and proof of the theorem is virtually unchanged when 
replacing $\psi_K(x;q,a)$ by $\theta_K(x;q,a)$.
\end{rmk}
In this article, we continue to assume that $K/\qq$ is Galois and 
replace the inequality in Theorem~\ref{bdh_gen} by an asymptotic formula.  
In particular, we show the following.
\begin{thm}\label{asymp}
For a fixed $M>0$,
\begin{equation}\label{up to x}
\sum_{q\le x}\sum_{a\in G_q}\left(
\theta_K(x;q,a)-\frac{x}{\varphi_K(q)}\right)^2
=[K:\qq]x^2\log x+C_1x^2
+O\left(\frac{x^2}{(\log x)^M}\right);
\end{equation}
and if $1\le Q\le x$,
\begin{equation}\label{up to Q}
\begin{split}
\sum_{q\le Q}\sum_{a\in G_q}\left(
\theta_K(x;q,a)-\frac{x}{\varphi_K(q)}
\right)^2
=[K:\qq]xQ\log x&-
\frac{\varphi(m_K)}{\varphi_K(m_K)}
xQ\log(x/Q)+C_2Qx\\
 &+O\left(x^{3/4}Q^{5/4}
+\frac{x^2}{(\log x)^M}\right),
\end{split}
\end{equation}
where $\varphi$ denotes the ordinary Euler $\varphi$-function,
$C_1,C_2$ are constants, and $m_K$ is an integer defined in the first 
paragraph of Section~\ref{prelim}.
\end{thm}
\begin{rmk}
The constants $C_1,C_2$ appearing in the statement of the theorem depend 
on $K$ and may be 
given explicitly.  However, the expressions are somewhat messy.
For example, $C_1$ is given by
\begin{equation*}
C_1=F(1)\zeta'(2)+F(1)\frac{(2\gamma-3)\pi^2}{12}+F(1)F'(1)\frac{\pi^2}{6}
-[K:\qq].
\end{equation*}
Here, $\zeta(s)$ denotes the Riemann zeta function,
$\gamma\approx 0.577$ is the Euler-Mascheroni constant, and
$F(s):=h(s)\prod_{\ell|m_K}D_{K,\ell}(s)$. The functions 
$h(s)$ and $D_{K,\ell}(s)$ are 
described in Section~\ref{prelim}.
\end{rmk}
\begin{rmk}
In the case that $K/\qq$ is Abelian, it turns out that 
$\varphi(m_K)/\varphi_K(m_K)=[K:\qq]$.  See the first paragraph of 
Section~\ref{prelim}.  Thus, in this case, equation~\eqref{up to Q} 
simplifies nicely to
\begin{equation*}
\begin{split}
\sum_{q\le Q}\sum_{a\in G_q}\left(
\theta_K(x;q,a)-\frac{x}{\varphi_K(q)}
\right)^2
=[K:\qq]xQ\log Q
+C_2Qx
+O\left(x^{3/4}Q^{5/4}
+\frac{x^2}{(\log x)^M}\right).
\end{split}
\end{equation*}
\end{rmk}
Our proof of Theorem~\ref{asymp} is an adaptation of Hooley's methods for the 
case $K=\qq$ as found in~\cite[pp. 209-212]{Hoo:1975}.
The proof will be carried out in Section~\ref{proof}.

\section{Acknowledgment}

The author is grateful to Andrew Granville both for suggesting this work and 
for pointing him toward the paper of Hooley~\cite{Hoo:1975} which was so very 
helpful. 

\section{Preliminaries}\label{prelim}

Before proceeding with the proof of Theorem~\ref{asymp}, we first
analyze the arithmetic function $\varphi_K(q)$.
Let $\qq^{\text{cyc}}:=\bigcup_{q>1}\qq(\zeta_q)$, and let $\mathcal{A}
:=\qq^{\text{cyc}}\cap K$.  Then $\A$ is an Abelian extension of $\qq$ 
of finite degree.  In particular, $\A$ is the maximal Abelian subfield of $K$. 
By the Kronecker-Weber Theorem,
there exists a smallest
integer $m_K$ such that $\mathcal{A}\subseteq\qq(\zeta_{m_K})$ 
See, for example,~\cite[p. 210]{Lan:1994}.
For each integer $q>0$, we define the intersection
$A_q:=K\cap\qq(\zeta_q)$.  Whence, via restriction maps,
$\Gal(K(\zeta_q)/K)\isom\Gal(\qq(\zeta_q)/A_q)$.
Thus, it is clear that if $q$ is coprime to $m_K$, then 
$\varphi_K(q)=\varphi(q)$.
In any case, $\varphi_K(q)$ is multiplicative and divides $\varphi(q)$.
For each prime divisor $\ell$ of $m_K$, we define $b_\ell:=\ord_\ell(m_K)$,
the order of $\ell$ dividing $m_K$.
\begin{lma}\label{prod formula for phi_K}
For a prime $\ell$, $\varphi_K(\ell)$ is a divisor of $\ell-1$.
In general, we have
\begin{equation*}
\varphi_K(q)=
\prod_{\substack{\ell^\alpha||q\\ \ell\dnd m_K}}\ell^{\alpha-1}(\ell-1)
\prod_{\substack{\ell^\alpha||q\\ \ell|m_K\\ \alpha\ge b_\ell}}
\ell^{\alpha-b_\ell}
\varphi_K(\ell)
\prod_{\substack{\ell^\alpha||q\\ \ell|m_K\\ \alpha<b_\ell}}
\varphi_K(\ell).
\end{equation*}
\end{lma}
\begin{proof}
The first statement is trivial as $G_q$ is a subgroup of $(\zz/q\zz)^*$.
Since $\varphi_K(q)$ is multiplicative and $\varphi_K(q)=\varphi(q)$ for 
$\gcd(q,m_K)=1$, we restrict attention to primes dividing $m_K$.  

Suppose that $\ell$ is a prime dividing 
$m_K$.  Then $A_{\ell^{b_\ell+k}}=A_{\ell^{b_\ell}}$
for all integers $k\ge 0$.  Thus, we immediately see that
\begin{equation}\label{exp larger than b}
\varphi_K(\ell^{b_\ell+k})
=|\Gal(\qq(\zeta_{\ell^{b_\ell+k}})/\qq(\zeta_{\ell^{b_\ell}}))|\cdot
|\Gal(\qq(\zeta_{\ell^{b_\ell}})/A_{\ell^{b_\ell}})|
=\ell^k\varphi_K(\ell^{b_\ell}).
\end{equation}
We claim that 
\begin{equation}\label{exp smaller than b}
\varphi_K(\ell^j)=\varphi_K(\ell) \text{ for }1\le j\le b_\ell.
\end{equation}
If $b_\ell=1$, the statement is trivial. Assume then that $b_\ell\ge 2$, and
consider the following field diagram.
\begin{equation}\label{compositum diagram}
\begin{diagram}
\node{}
\node{\qq(\zeta_{\ell^{b_\ell}})}\arrow{se,t,-}{\varphi_K(\ell^{b_\ell})}
\arrow{sw,t,-}{\ell^{b_\ell-1}}
\node{}\\
\node{\qq(\zeta_{\ell})}\arrow{se,b,-}{\varphi_K(\ell)}
\node{}
\node{A_{\ell^{b_\ell}}}\arrow{sw,-}\\
\node{}\node{A_\ell}\node{}
\end{diagram}
\end{equation}
Observe that 
$A_\ell=K\cap\qq(\zeta_\ell)=A_{\ell^{b_\ell}}\cap\qq(\zeta_\ell)$.
Since the compositum $A_{\ell^{b_\ell}}\qq(\zeta_{\ell})$ is the 
smallest field containing both $A_{\ell^{b_\ell}}$ and $\qq(\zeta_{\ell})$, 
we have that
$\qq(\zeta_{\ell^{b_\ell}})\supseteq 
A_{\ell^{b_\ell}}\qq(\zeta_{\ell})
\supseteq\qq(\zeta_{\ell})$.
The Galois group $\Gal(\qq(\zeta_{\ell^{b_\ell}})/\qq(\zeta_\ell))$ is 
cyclic of order $\ell^{b_\ell-1}$.  
We deduce then that 
$A_{\ell^{b_\ell}}\qq(\zeta_{\ell})=\qq(\zeta_{\ell^{j_0}})$ 
for some $1\le j_0\le b_\ell$.
However, since $m_K$ is minimal, $b_\ell$ must be minimal as well.
Therefore, we must have that
$A_{\ell^{b_\ell}}\not\subseteq\qq(\zeta_{\ell^{b_\ell-1}})$.
This implies that 
$A_{\ell^{b_\ell}}\qq(\zeta_{\ell})=\qq(\zeta_{\ell^{b_\ell}})$.
Thus, from the diagram~\eqref{compositum diagram}, we see that
$\varphi_K(\ell)=\varphi_K(\ell^{b_\ell})$.
The claim in~\eqref{exp smaller than b} follows since 
$\varphi_K(\ell^j)$ divides $\varphi_K(\ell^{j+1})$ for all $j\ge 1$.
The lemma follows by combining~\eqref{exp larger than b} 
with~\eqref{exp smaller than b}.
\end{proof}

The final goal of this section is to study the Dirichlet generating function
\begin{equation*}
D_{K}(s):=\sum_{n=1}^\infty\frac{1}{\varphi_K(n)n^{s-1}}
\end{equation*}
and use it to prove two asymptotic identities involving the function 
$\varphi_K(n)$.
Since $\varphi_K(n)$ agrees with $\varphi(n)$ for
$\gcd(n,m_K)=1$, we begin with the Dirichlet series
\begin{equation*}
D(s):=\sum_{n=1}^\infty\frac{1}{\varphi(n)n^{s-1}}
\end{equation*}
and introduce finitely many correction factors to obtain 
$D_K(s)$.
Let $h(s)$ denote the Euler product
\begin{equation*}
h(s):=\prod_{\ell}\left\{1+\frac{1}{\ell^{s+2}}\left(1-\frac{1}{\ell^s}\right)
\left(1-\frac{1}{\ell}\right)^{-1}\right\};
\end{equation*}
and observe that, for any $\epsilon>0$, $h(s)$ is holomorphic and bounded 
for $\text{Re}(s)>-\frac{1}{2}+\epsilon$.
Using the product formula for Euler's $\varphi$ function, we factor 
$D(s)$ as
\begin{equation}
\begin{split}
D(s)&=\prod_{\ell}\left\{
1+\frac{1}{\ell^s}\left(
1-\frac{1}{\ell}\right)^{-1}\left(1-\frac{1}{\ell^s}\right)^{-1}
\right\}
=\zeta(s)\zeta(s+1)h(s),\label{D factored}
\end{split}
\end{equation}
where again $\zeta(s)$ is the Riemann zeta function.

We now return to the Dirichlet series $D_K(s)$.
In light of~\eqref{D factored} and Lemma~\ref{prod formula for phi_K},
for each prime $\ell$ dividing $m_K$, we define the correction factor
\begin{equation*}
D_{K,\ell}(s)
:=\frac{
	\left\{\displaystyle 1+
\frac{1}{\varphi_K(\ell)\ell^{s-1}}\left(1-\left(\frac{1}{\ell^{s-1}}\right)^{b_\ell-1}\right)\left(1-\frac{1}{\ell^{s-1}}\right)^{-1}
		+\frac{1}{\varphi_K(\ell)}
		\left(\frac{1}{\ell^{s-1}}\right)^{b_\ell}
	\left(1-\frac{1}{\ell^s}\right)^{-1}
	\right\}
}
{\displaystyle\left\{1+\frac{1}{\ell^s}
\left(1-\frac{1}{\ell}\right)^{-1}\left(1-\frac{1}{\ell^s}\right)^{-1}\right\}},
\end{equation*}
which has removable singularities at $s=0,1$ and is analytic 
elsewhere.
We also define $D_{K,\ell}(0)$ (resp. $D_{K,\ell}(1)$) to be the limit 
of $D_{K,\ell}(s)$ as $s$ approaches $0$ (resp. $1$).
In particular, we note that
\begin{equation}\label{limit}
D_{K,\ell}(0)=\lim_{s\rightarrow 0}D_{K,\ell}(s)
=\frac{\varphi(\ell^{b_\ell})}{\varphi_K(\ell)}
=\frac{\varphi(\ell^{b_\ell})}{\varphi_K(\ell^{b_\ell})}.
\end{equation}
Finally, from~\eqref{D factored},
we observe that $D_K(s)$ may be factored as
\begin{equation}\label{D_K factored}
D_K(s)=\zeta(s)\zeta(s+1)h(s)\prod_{\ell|m_K}D_{K,\ell}(s).
\end{equation}
\begin{lma}\label{lma1}
For a fixed number field $K$, we have
\begin{align}
\sum_{n<x}\left(1-\frac{n}{x}\right)^2\frac{1}{\varphi_K(n)}
&=c_{1}\log x+c_2+
\frac{\varphi(m_K)}{\varphi_K(m_K)}\frac{\log x}{x}
+\frac{c_3}{x}+O\left(x^{-\frac{5}{4}}\right)\label{id1};\\
\sum_{n\le x}\frac{1}{\varphi_K(n)}
&=c_1\log x+c_4+O\left(\frac{1}{x}\right),
\label{id2}
\end{align}
where $c_1=
\frac{\zeta(2)\zeta(3)}{\zeta(6)}
\prod_{\ell| m_K}D_{K,\ell}(1),$ and $c_2,c_3,c_4$ are constants.
\end{lma}
\begin{proof}
We begin with the proof of~\eqref{id1}.
For $c>0$,
\begin{align*}
\frac{1}{2}\sum_{n<x}\left(1-\frac{n}{x}\right)^2\frac{1}{\varphi_K(n)}
&=\frac{1}{2\pi i}\int_{c-i\infty}^{c+i\infty}D_K(s+1)\frac{x^s}{s(s+1)(s+2)}ds\\
&=R_0+R_{-1}
+\frac{1}{2\pi i}
\int_{-\frac{5}{4}-i\infty}^{-\frac{5}{4}+i\infty}
D_K(s+1)\frac{x^s}{s(s+1)(s+2)}ds,
\end{align*}
where $R_0$ and $R_{-1}$ are the residues of the integrand 
at $s=0$ and $s=-1$ respectively.
See~\cite[Exercise 4.1.9, p. 57]{Mur:2001} for example.
Using~\eqref{D_K factored}, we calculate the residues as follows:
\begin{align*}
R_0&=\frac{\zeta(2)h(1)\prod_{\ell|m_K}D_{K,\ell}(1)}{2}\log x+\frac{1}{2}c_2
=\frac{c_1}{2}\log x+\frac{1}{2}c_2;\\
R_{-1}&=\frac{-\zeta(0)h(0)\prod_{\ell|m_K}D_{K,\ell}(0)\log x}{x}
+\frac{c_3}{2x}
=\frac{\varphi(m_K)}{\varphi_K(m_K)}
\frac{\log x}{2x}+\frac{c_3}{2x},
\end{align*}
where we have applied~\eqref{limit} to compute
$\prod_{\ell|m_K}D_{K,\ell}(0)$
The remaining integral is clearly $O(x^{-5/4})$.

For the proof of~\eqref{id2}, we begin with the formula
\begin{align*}
 \sum_{n\le x}\frac{1}{\varphi_K(n)}
&=\frac{1}{2\pi i}\int_{c-i\infty}^{c+i\infty}D_K(s+1)\frac{x^s}{s}ds
\end{align*}
and proceed in a manner similar to the proof of~\eqref{id1}.
\end{proof}

\section{Proof of Theorem~\ref{asymp}}\label{proof}

Let $\displaystyle\theta_K(x):=\sum_{\N\p\le x}\log\N\p$.  
We will frequently make use of the formula
\begin{equation}\label{KPNT}
\theta_K(x)=x+O(x/(\log x)^M)
\end{equation} 
throughout the remainder of the article.
The formula follows from~\eqref{gen SW}.
We now begin the proof of Theorem~\ref{asymp} by stating and proving 
the following lemma.

\begin{lma}\label{H lma}
For any $M>0$,
\begin{equation*}
\sum_{\N\p\le x}\sum_{\N\p'=\N\p}(\log\N\p)^2
=[K:\qq](x\log x-x)+O\left(\frac{x}{(\log x)^M}\right).
\end{equation*}
\end{lma}
\begin{proof}
 First, note that since only finitely many rational primes may ramify in $K$, 
we only introduce an error of $O(1)$ by restricting our sum to prime ideals
which do not lie above a rational prime ramifying in $K$.  For a rational 
prime $p$, let $g_p$ denote the number of primes lying above $p$, 
let $f_p$ denote the degree of any prime lying above $p$, and 
let $e_p$ denote the ramification index of $p$ in $K$.
Note that $e_p$ and $f_p$ are well-defined since $K/\qq$ is Galois. 
The contribution from the degree one primes gives us our main term.  
Thus, partial summation and~\eqref{KPNT} yield
\begin{align*}
\sum_{\N\p\le x}\sum_{\N\p'=\N\p}(\log\N\p)^2
&=[K:\qq]\sum_{\substack{p\le x\\ e_p=1\\ f_p=1}}g_p(\log p)^2+O(\sqrt x\log x)\\
&=[K:\qq]\log x\left(\theta_K(x)+O(\sqrt x)\right)
-[K:\qq]\int_1^x\frac{\theta_K(t)+O(\sqrt t)}{t}dt\\
&=[K:\qq](x\log x-x)+O\left(x(\log x)^{-M}\right).
\end{align*}
\end{proof}

\begin{proof}[Proof of Theorem~\ref{asymp}]
First, define
\begin{equation*}
S(x;Q_1,Q_2):=\sum_{Q_1<q\le Q_2}
\sum_{a\in G_q}
\left(
\theta_K(x;q,a)-\frac{x}{\varphi_K(q)}
\right)^2.
\end{equation*}
If $Q\le x(\log x)^{-(M+1)}$, then Theorem~\ref{bdh_gen} implies that 
$S(x;0,Q)\ll x^2(\log x)^{-M}$, and hence Theorem~\ref{asymp} follows since 
the error term dominates in this case.  Thus, it suffices to consider 
the case when $Q>x(\log x)^{-(M+1)}$.  
Therefore, for the remainder of the proof, 
we will write $Q_1:=x(\log x)^{-(M+1)}$, and assume that 
$Q_1<Q_2\le x$.
By Theorem~\ref{bdh_gen}, we have
\begin{equation}\label{sieve}
S(x;0,Q_2)=S(x;Q_1,Q_2)+O\left(x^2(\log x)^{-M}\right).
\end{equation}
For $Q_1,Q_2$ as above,
\begin{align}
S(x;Q_1,Q_2)&=
\sum_{Q_1<q\le Q_2}
   \sum_{a\in G_q}
    \left\{
      \theta_K(x;q,a)^2-\frac{2x}{\varphi_K(q)}\theta_K(x;q,a)+\frac{x^2}{\varphi_K(q)^2}
    \right\}\nonumber \\
&=\sum_{Q_1<q\le Q_2}
   \left\{
     \sum_{a\in G_q}\theta_K(x;q,a)^2
       -\frac{x}{\varphi_K(q)}\left(2\theta_K(x)
         -2\sum_{\substack{\N\p\le x,\\(\N\p,q)>1}}\log\N\p-x\right)
   \right\}\nonumber \\
&=\sum_{Q_1<q\le Q_2}
     \sum_{a\in G_q}\theta_K(x;q,a)^2
  -x^2\sum_{Q_1<q\le Q_2}\frac{1}{\varphi_K(q)}
  +O\left(\frac{x^2}{(\log x)^M}\right).\label{first simp}
\end{align}
Now, observe that
\begin{align*}
\sum_{a\in G_q}\theta_K(x;q,a)^2
&=
\sum_{\substack{\N\p,\N\p'\le x,\\ 
\N\p\equiv\N\p'\pmod{q},\nonumber \\ 
(\p\p',q\ints_K)=1}}\log\N\p\log\N\p'\\
&=\sum_{\substack{\N\p=\N\p'\le x,\\ (\p\p',q\ints_K)=1}}(\log\N\p)^2
+\sum_{\substack{\N\p,\N\p'\le x;\ \N\p\neq\N\p',\\ \N\p\equiv\N\p'\pmod q}}
\log\N\p\log\N\p'.
\end{align*}
Note that removing the condition $(\p\p',q\ints_K)=1$ from the second sum is justified.
For example, if $\p|q\ints_K$ and $p$ lies below $\p$, 
then the condition $\N\p\equiv\N\p'\pmod q$
implies that $0\equiv\N\p'\pmod p$.  This in turn implies that $\N\p=\N\p'$.
Thus, we define
\begin{align*}
 H(x;Q_1,Q_2)&:=\sum_{Q_1<q\le Q_2}
\sum_{\substack{\N\p=\N\p'\le x,\\ (\p\p',q\ints_K)=1}}(\log\N\p)^2;\\
J(x;Q_1,Q_2)&:=\sum_{Q_1<q\le Q_2}
\sum_{\substack{\N\p,\N\p'\le x;\ \N\p\neq\N\p',\\ \N\p\equiv\N\p'\pmod q}}
\log\N\p\log\N\p'.
\end{align*}
Now~\eqref{first simp} may be rewritten as
\begin{align}
 S(x;Q_1,Q_2)=H(x;Q_1,Q_2)&+J(x;Q_1,Q_2)\nonumber\\
&-c_1x^2\log(Q_2/Q_1)
+O\left(\frac{x^2}{(\log x)^M}\right).\label{second simp}
\end{align}
Note that we have applied the second part of Lemma~\ref{lma1} to the second term
of~\eqref{first simp}.

Removing the condition $(\p\p',q\ints_K)=1$ from the inner sum of 
$H(x;Q_1,Q_2)$ introduces an error which is $O\left((\log x)^2\right)$.
Thus, we may apply Lemma~\ref{H lma} to obtain
\begin{align}
H(x;Q_1,Q_2)&=\left\{Q_2-Q_1+O(1)\right\}
\left\{[K:\qq](x\log x-x)+O\left(x(\log x)^{-M}\right)\right\}\nonumber\\
&=[K:\qq]xQ_2\log x-[K:\qq]xQ_2+O\left(\frac{x^2}{(\log x)^{M}}\right).\label{H simp}
\end{align}

Now, define $J(x;Q):=J(x;Q,x)$, so that $J(x;Q_1,Q_2)=J(x;Q_1)-J(x;Q_2)$.
Then
\begin{align*}
J(x;Q)&=2
\sum_{\substack{\N\p'<\N\p\le x,\\ \N\p-\N\p'=kq,\\ Q<q\le x}}
\log\N\p\log\N\p'
=2\sum_{k<x/Q}
\sum_{\substack{\N\p\equiv\N\p'\pmod k,\\ \ \N\p\le x;\N\p-\N\p'>kQ}}
\log\N\p\log\N\p'\\
&=2\sum_{k<x/Q}\sum_{a\in G_k}
\sum_{\substack{\N\p'<x-kQ,\\ \N\p'\equiv a\pmod k}}\log\N\p'
\sum_{\substack{kQ+\N\p'<\N\p\le x,\\ \N\p\equiv a\pmod k}}\log\N\p.
\end{align*}
Since $Q\ge Q_1=x/(\log x)^{M+1}$, we have
$k<x/Q\le (\log x)^{M+1}$ and $kQ\ge x/(\log x)^{M+1}$.
Thus, we may apply~\eqref{gen SW} and write 
\begin{equation*}
 \theta_K(x;a,k)-\theta_K(kQ+\N\p';a,k)
=\frac{x-kQ-\N\p'}{\varphi_K(k)}
+O\left(\frac{x}{(\log x)^{2M+1}}\right)
\end{equation*}
for the innermost sum above.
This gives
\begin{align*}
 J(x,Q)&=
2\sum_{k<\frac{x}{Q}}\frac{1}{\varphi_K(k)}
\sum_{\substack{\N\p'<x-kQ,\\ (\N\p',k)=1}}(x-kQ-\N\p')\log\N\p'
+O\left(\frac{x}{(\log x)^{2M+1}}\sum_{k<\frac{x}{Q}}\theta_K(x)\right)\\
&=2\sum_{k<\frac{x}{Q}}\frac{\int_1^{x-kQ}\theta_K(t)dt}{\varphi_K(k)}
+O\left(x\sum_{k<\frac{x}{Q}}\frac{\log k}{\varphi_K(k)}\right)
+O\left(\frac{x^3}{Q(\log x)^{2M+1}}\right),
\end{align*}
where the last line follows by partial summation applied to 
the inner sum of the main term.
Therefore, by~\eqref{KPNT}, we have
\begin{equation*}
 J(x,Q)=x^2\sum_{k<\frac{x}{Q}}\left(1-\frac{kQ}{x}\right)^2\frac{1}{\varphi_K(k)}
+O\left(\frac{x^2}{(\log x)^M}\right).
\end{equation*}

We consider two different cases for the treatment of $J(x;Q_1,Q_2)$.
First, if $Q_2=x$, then
\begin{align}
 J(x;Q_1,Q_2)&=J(x;Q_1)\nonumber\\
&=x^2\left\{c_1\log(x/Q_1)+c_2
+O\left(\frac{\log(x/Q_1)}{x/Q_1}\right)\right\}
+O\left(\frac{x^2}{(\log x)^M}\right)\nonumber\\
&=c_1x^2\log(Q_2/Q_1)+c_2x^2
+O\left(\frac{x^2}{(\log x)^M}\right).\label{case x}
\end{align}
In the case that $Q_2\le x$ (including the previous case), we may write
\begin{align}
 J(x;Q_1,Q_2)=&J(x;Q_1)-J(x;Q_2)\nonumber\\
=&c_1x^2\log(Q_2/Q_1)
-\frac{\varphi(m_K)}{\varphi_K(m_K)}xQ_2\log(x/Q_2)-c_3xQ_2\nonumber\\
&+O\left(x^{3/4}Q_2^{5/4}\right)
+O\left(\frac{x^2}{(\log x)^M}\right).\label{case le x}
\end{align}
Theorem~\ref{asymp} now follows by 
combining~\eqref{sieve}, \eqref{second simp}, \eqref{H simp},
\eqref{case x}, and~\eqref{case le x}.
\end{proof}

\bibliographystyle{plain}
\bibliography{references}
\end{document}